\documentclass[preprint,12pt,1p]{elsarticle}

\usepackage{amssymb}
\usepackage{graphicx}
\usepackage[all]{xy}
\usepackage{amsmath}
\usepackage{mathtools}
\usepackage{tikz-cd}

\newcommand{\too}{\longrightarrow}

\newcommand{\Om}{\Omega}
\newcommand{\na}{\nabla}

\newcommand{\al}{\alpha}
\newcommand{\be}{\beta}

\newcommand{\ga}{\gamma}
\newcommand{\Ga}{\Gamma}

\font\bb=msbm10

\def\R{\hbox{\bb R}}

\newtheorem{theorem}{Theorem}[section]
\newtheorem{lemma}[theorem]{Lemma}
\newtheorem{definition}[theorem]{Definition}
\newtheorem{example}[theorem]{Example}

\newtheorem{proposition}[theorem]{Proposition}
\newtheorem{remark}[theorem]{Remark}
\newtheorem{corollary}[theorem]{Corollary}
\numberwithin{equation}{section}
\newcommand{\dreqno}{\let\veqno\eqno}

\makeatletter
\DeclareRobustCommand{\qed}{%
	\ifmmode 
	\else \leavevmode\unskip\penalty9999 \hbox{}\nobreak\hfill
	\fi
	\quad\hbox{\qedsymbol}}
\newcommand{\openbox}{\leavevmode
	\hbox to.77778em{%
		\hfil\vrule
		\vbox to.675em{\hrule width.6em\vfil\hrule}%
		\vrule\hfil}}
\newcommand{\qedsymbol}{\openbox}
\newenvironment{proof}[1][\proofname]{\par
	\normalfont
	\topsep6\p@\@plus6\p@ \trivlist
	\item[\hskip\labelsep\itshape
	#1.]\ignorespaces
}{%
	\qed\endtrivlist
}
\newcommand{\proofname}{Proof}
\makeatother

\journal{}

\begin{document}

\begin{frontmatter}

\title{Submanifolds in Koszul-Vinberg geometry\vskip 0.5cm \today}


\author[label1]{Abdelhak Abouqateb}
\address[label1]{Universit$\acute{e}$ Cadi-Ayyad
	Facult$\acute{e}$ des sciences et techniques
	BP 549 Marrakech Maroc}


\ead{a.abouqateb@uca.ac.ma}

\author[label1]{Mohamed Boucetta}
\ead{m.boucetta@uca.ac.ma}

\author[label1]{Charif Bourzik}
\ead{bourzikcharif@gmail.com}
\begin{abstract} 
	
A Koszul-Vinberg manifold is a manifold $M$ endowed with a pair $(\nabla,h)$ where $\nabla$ is a flat connection and $h$ is a symmetric bivector field satisfying a generalized Codazzi equation. The geometry of such manifolds could be seen as a type of bridge between Poisson geometry and pseudo-Riemannian geometry, as has been highlighted in our previous article [\textit{Contravariant Pseudo-Hessian manifolds and their associated Poisson structures}. \rm{Differential Geometry and its Applications} (2020)]. Our objective here will be to pursue our study by focusing in this setting on submanifolds by taking into account some developments in the theory of Poisson submanifolds. 
\end{abstract}

\end{frontmatter}


\section{Declarations}
 Not applicable
\section*{Introduction}
\label{sec1}
An $n$-dimensional manifold $M$ is called an affine manifold when it is equipped with a flat connection $\nabla$ on the tangent bundle $TM\rightarrow M$ or, equivalently, the manifold $M$ is equipped with a maximal atlas such that all transition functions are restrictions of affine transformations of $\R^n$. For such manifolds, there exists around every point an affine coordinates system, i.e. a coordinates system $(x_{1},\ldots,x_{n})$ satisfying $\nabla\partial_{x_{i}}=0$ for any $i=1,\ldots,n$ (see\cite{Shima}). Now let $(M,\nabla)$ be an affine manifold of dimension $n$, $h$ be a symmetric bivector field on $M$ and $h_{\#}:T^{*}M\rightarrow TM$ the associated contraction given by $\beta(\alpha^{\#})=h(\alpha,\beta)$ where $\alpha^{\#}:=h_{\#}(\alpha)$. We recall (\cite{ABB1}) that $h$ is said to be a {\it{pseudo-Hessian bivector field}} on $(M,\nabla)$ if it satisfies the {\it{contravariant Codazzi equation}}
\begin{equation}\label{codazzi}
(\nabla_{\alpha^{\#}}h)(\beta,\gamma)=(\nabla_{\beta^{\#}}h)(\alpha,\gamma),\quad
\end{equation}
for any $\alpha,\beta,\gamma\in\Omega^1(M)$. The triple $(M,\nabla,h)$ is then called a {\it{ contravariant pseudo-Hessian manifold}}. This means that around each affine coordinates system $(x_1,\ldots,x_n)$, we have $(dx_i)^{\#}(h_{jk})=(dx_j)^{\#}(h_{ik})$ for any $i,j,k=1,\ldots,n$; which could be written as follows 
\begin{equation}\label{eq2}
\displaystyle\sum_{l=1}^{n}\left(h_{il}\dfrac{\partial h_{jk}}{\partial x_l}-h_{jl}\dfrac{\partial h_{ik}}{\partial x_l} \right)=0,
\end{equation}
where $h_{ij}=h(dx_i,dx_j)$.

Motivated by the theory of left-symmetric algebroid the authors in \cite{Wang} introduce the notion of Koszul-Vinberg structure on a left-symmetric algebroid. Lets clarify briefly this fact: Let $(A,\bullet,\rho)$ be a left symmetric algebroid and $h$ a symmetric bivetor field on $A$ then $h$ is called a Koszul-Vinberg structure on $A$ if and only if $[h,h]=0$ where:
\begin{eqnarray*}
	[h,h](\al,\be,\ga)&=&\rho(\al^{\#}).h(\be,\ga)-\rho(\be^{\#}).h(\al,\ga)+\prec \al,\be^{\#}\bullet\ga^{\#}\succ\\
	&  &-\prec\be,\al^{\#}\bullet\ga^{\#}\succ-\prec\ga,[\al^{\#},\be^{\#}]_{A}\succ.
\end{eqnarray*}
 The notion contravariant pseudo-Hessian manifolds could be seen as a Koszul-Vinberg structure on the left symmetric algebroid $A=(TM,\na,id)$. For this reason all notions introduced under the name contravariant pseudo-Hessian in \cite{ABB1} and \cite{Boucetta1} will be called Koszul-Vinberg. 

Keeping in mind the analogies given in \cite{ABB1} between Koszul-Vinberg manifolds and Poisson manifolds it is natural to ask if we can define an analogue of Hamiltonian vector fields, this leads to a new questions and some important geometric and algebraic proprieties. It is also important to study the maps which preserves this structure which are the analogue of Poisson maps, in particular we show that there is a correspondence between such maps and Poisson maps. Taking into account the study devoted to submanifolds in Poisson geometry see for example the survey paper \cite{Zambon} we introduce the notion of Koszul-Vinberg submanifolds which are the analogue of Poisson submanifolds, and we show that a Koszul-Vinberg manifold in some sense is the union of the affine leaves which meet such submanifolds. Also we introduce a notion  of transversal submanifolds which are the analogue of tranversal submanifolds in Poisson geometry known also under the name cosymplectic submanifolds, and we show that they inherit naturally a Koszul-Vinberg structure. Also we introduce a notion coisotropic Koszul-Vinberg submanifolds which are the analogue of coisotropic submanifolds in Poisson geometry (see \cite{Weinstein}). It is well known that Poisson maps can be characterized using coisotropic submanifolds in Poisson geometry (see \cite{Weinstein}), so we show that the same thing happens in Koszul-Vinberg geometry. Also we know that the conormal bundle of coisotropic submanifolds inherited naturally a Lie algebroid structure, in this direction we show that the conormal bundle of a coisotropic Koszul-Vinberg submanifolds is a left-symmetric algebroid.  
\vskip 0.2cm
The paper is structured as follows. In Section $1$ we give some preliminaries and basic properties developed in \cite{ABB1} about Koszul-Vinberg manifolds.
In Section $2$ we study Koszul-Vinberg Hamiltonian vector fields and we explore their related properties. Section $3$ is devoted to the study of Koszul-Vinberg maps and their relationships with Poisson maps. The study of Koszul-Vinberg submanifolds, as a similarity with Poisson submanifolds, is the main purpose of  Section $4$. After we move to the description of Koszul-Vinberg traversals in section $5$. The last section deals with the notion of coisotropic Koszul-Vinberg submanifolds.

\section{Preliminaries}

Let us recall from \cite{ABB1} that  Koszul-Vinberg manifolds define a subclass of Lie algebroids in the following way: Let $(M,\nabla,h)$ be an affine manifold endowed with a symmetric bivector field. We associate to this triple a bracket on $\Om^1(M)$ by putting
\begin{equation}\label{eq 0}
{[\alpha,\beta]_h}:=\nabla_{\alpha^{\#}}\beta-\nabla_{\beta^{\#}}\alpha.
\end{equation}  
and a contravariant connection $\mathcal{D}:\Om^{1}(M)\times\Om^{1}(M)\rightarrow\Om^{1}(M)$ given by
\begin{eqnarray}
\prec \mathcal{D}_{\alpha}\beta,X\succ&=&(\na_{X}h)(\alpha,\beta)+\prec \na_{\al^\#}\beta,X\succ\label{eq 1}\\
&=&-X.h(\al,\be)+\al^{\#}.\prec \be,X\succ+\prec\al,\na_{X}\be^{\#}\succ\nonumber\\
& &+\prec\be,\left[X,\al^{\#}\right]\succ,\nonumber
\end{eqnarray} 
for any $\al,\be\in\Om^{1}(M)$ and $X\in\Ga(TM)$, satisfying the  relations
\begin{equation}\label{eq 2}
[\al,\be]_{h}=\mathcal{D}_{\al}\be-\mathcal{D}_{\be}\al\text{ and }[\al,f\be]_{h}=f[\al,\be]_{h}+\al^{\#}(f)\be, 
\end{equation}
where $f\in C^{\infty}(M)$.

 Theorem $2.4$ in \cite{ABB1} shows that $(M,\na,h)$ is a Koszul-Vinberg manifold if and only if $(T^*M,h_\#,[\;,\;]_h)$ is a Lie algebroid. Moreover, $\mathcal{D}$ is a connection for the Lie algebroid structure $(T^*M,h_\#,[\;,\;]_h)$ satisfying
\begin{equation}\label{eq 3}
(\mathcal{D}_{\al}\be)^{\#}=\na_{\al^\#}\be^{\#}.
\end{equation}
Also any leaf $L$ of the singular foliation on $M$ induced by the Lie algberoid structure support a pseudo-Hessian structure $(g_L,\na^{L})$ where $g_L(\al^\#,\be^\#)=h(\al,\be)$. This foliation will be called affine foliation. 

There is a more important geometric structures on the tangent bundle of $TM$. Indeed, associated to $\nabla$ there exists a splitting 
\begin{equation*}
TTM=V(M)\oplus H(M)
\end{equation*}
such that for any $u\in TM$, $T_{u}p:H_{u}(M)\rightarrow T_{p(u)}M$ is an isomorphism. For any $X\in\Gamma(TM)$ we denote by $X^{v}\in\Gamma(V(M))$ its vertical lift and by $X^{h}\in\Gamma(V(M))$ its horizontal lift. There are given, for any $u\in TM$, by
\begin{equation}\label{eq 5}
X^{v}_{u}=\frac{d}{dt}_{|_{t=0}}(u+tX_{p(u)}), \text{ and } T_{u}p(X^{h}_{u})=X_{p(u)}.
\end{equation} 
The vector $X^{h}_{u}$ can also be defined using parallel transport in the following way: Let $\gamma$ be a smooth curve on $M$ starting at $x$ and its derivative at $0$ is the vector $X_x$
\begin{equation}\label{eq 6}
X^h_u=\frac{d}{dt}_{\lvert_{t=0} }\tau^\gamma_{0t}(u)
\end{equation}	  
where $\tau^\gamma_{0t}:T_xM\stackrel{\cong}{\rightarrow} T_{\gamma(t)}M$ is the parallel transport map along $\gamma$.
The {\textit{Sasaki almost complex structure}} $J:TTM\rightarrow TTM$ determined by $\na$ is defined by 
\begin{equation}\label{eq 7}
J(X^{h})=X^{v}\text{ and } J(X^{v})=-X^{h}.
\end{equation} 
It is integrable to a complex structure on $TM$ if and only if $\nabla$ is flat.

The vanishing of the curvature of $\nabla$ implies that the Lie bracket on $\Ga(T(TM))$ is determined by:
\begin{equation}\label{eq 8}
[X^{h},Y^{h}]=[X,Y]^{h}, [X^{h},Y^{v}]=(\nabla_{X}Y)^{v}\text{ and } [X^{v},Y^{v}]=0.
\end{equation}
As for the vector fields, for any $\alpha\in\Omega^{1}(M)$, we define $\alpha^{v},\alpha^{h}\in\Omega^{1}(TM)$ by
\begin{equation*}
\left\{
\begin{array}{ll}
\prec \al^v,X^v\succ=\prec \al,X\succ \circ p \\
\prec \al^v,X^h\succ=0
\end{array}
\right. \text{ and }
\left\{
\begin{array}{ll}
\prec \al^h,X^h\succ=\prec \al,X\succ \circ p \\
\prec \al^h,X^v\succ=0,
\end{array}
\right.
\end{equation*} 
one can see that $\al^{h}=p^{*}\al$.

The {\textit{Sasaki connection}} $\overline{\na}$ on $TM$ determined by $\na$ is defined by 
\begin{equation}\label{eq 10}
\overline{\na}_{X^{h}}Y^{h}=(\nabla_{X}Y)^{h},  \overline{\na}_{X^{h}}Y^{v}=(\nabla_{X}Y)^{v}\text{ and } \overline{\na}_{X^{v}}Y^{h}= \overline{\na}_{X^{v}}Y^{v}=0,
\end{equation}
where $X,Y\in\Gamma(TM)$. This connection is torsionless and flat and hence defines an affine structure on $TM$. Moreover, the endomorphism vector field $J:TTM\rightarrow TTM$ is parallel with respect to $\overline{\na}$.

 To any pair $(\na,h)$ we can associate a skew-symmetric bivector field $\Pi$ on $TM$ by putting 
\begin{equation}\label{eq 11}
\Pi(\alpha^{v},\beta^{v})= \Pi(\alpha^{h},\beta^{h})=0 \text{  and  } \Pi(\alpha^{h},\beta^{v})=- \Pi(\beta^{v},\beta^{h})=h(\alpha,\beta)\circ p,
\end{equation}
for any $\alpha,\beta\in\Om^{1}(M)$. So we get that  
\begin{equation}\label{eq 12}
\Pi_{\#}(\alpha^{v})=-(\alpha^{\#})^{h}\text{ and }\Pi_{\#}(\alpha^{h})=(\alpha^{\#})^{v}.
\end{equation}
 Theorem $3.3$ in \cite{ABB1} assert that $(M,\nabla,h)$ is a Koszul-Vinberg manifold if and only if $(TM,\Pi)$ is a Poisson manifold.
\vskip 0.5cm
\section{Koszul-Vinberg Hamiltonian vector fields}

 In Poisson geometry, Hamiltonian vector fields play an important role (see \cite{Vaisman}) and it is therefore natural to discuss some properties of their analogue in Koszul-Vinberg geometry. Let $(M,\nabla,h)$ be a K-V manifold. For any $f\in C^{\infty}(M)$ we can associate the vector field $X_{f}:=(df)^{\#}$, which will be called the {\textit{K-V Hamiltonian vector field}} associated to $f$. The symmetry of the bivector field $h$ leads to the following:
 \begin{equation*}
  X_{f_1}(f_2)=X_{f_2}(f_1)=\prec df_2,X_{f_1}\succ=\prec df_1,X_{f_2}\succ,\;\;\quad \forall f_1,f_2\in C^{\infty}(M).
 \end{equation*}
 
 Contrary to what happens in Poisson geometry the flow of the vector field $X_f$ does not generally preserve the K-V bivector field $h$, this can be seen through the following example.
 \begin{example}
 Consider the K-V manifold $M=(\mathbb{R}^{2},\na,h)$ where $\na$ is the canonical affine connection and $h=x\partial_{x}\otimes\partial_{x}+y\partial_{y}\otimes\partial_{y}$. Let $f:\mathbb{R}^{2}\rightarrow\mathbb{R},(x,y)\mapsto x$, by a direct computation we get that $L_{X_{f}}(h)(dx,dx)=-f$.
 \end{example}

More generally, we have:
\begin{proposition}\label{Pro 0}
	For any $f\in C^{\infty}(M)$ and any  $\al,\be\in\Om^{1}(M)$ we have 
	\begin{equation*}
	\mathcal{L}_{X_{f}}(h)(\al,\be)=-\na_{X_{f}}(h)(\al,\be)+2\prec\na_{\al^{\#}}df,\be^{\#}\succ.
	\end{equation*}
\end{proposition}
\begin{proof}
	Let $(x_{1},\ldots,x_{n})$ be an affine local coordinates system on $M$. Denote by $X_{i}=(dx_{i})^{\#}$. Then we have:
	\begin{eqnarray*}
		\mathcal{L}_{X_{f}}(h)(dx_{i},dx_{j})&=& X_{f}.h_{ij}-h\left(\mathcal{L}_{X_{f}}dx_{i},dx_{j}\right)-h\left(dx_{i},\mathcal{L}_{X_{f}}dx_{j}\right)\\
		&=& -X_{f}.h_{ij}+\prec dx_{i},[X_{f},X_{j}]\succ+\prec dx_{j},[X_{f},X_{i}]\succ\\
		&=& -X_{f}.h_{ij}+\prec dx_{i},[df,dx_{j}]^{\#}\succ+\prec dx_{j},[df,dx_i]^{\#}\succ\\
		&=& -X_{f}.h_{ij}+\prec [df,dx_{j}],X_{i}\succ+\prec [df,dx_{i}],X_{j}\succ\\
		&=& -X_{f}.h_{ij}+\prec \na_{X_{j}}df,X_{i}\succ-\prec \na_{X_{f}}dx_{j},X_{i}\succ\\
		& &+\prec \na_{X_{i}}df,X_{j}\succ-\prec \na_{X_{f}}dx_{i},X_{j}\succ.
	\end{eqnarray*}
   Now since $\na dx_{i}=0$ and $\prec \na_{X}df,Y\succ=\prec \na_{Y}df,X\succ$, we get 
   \begin{equation*}
   \mathcal{L}_{X_{f}}(h)(dx_{i},dx_{j})=-\na_{X_{f}}(h)(dx_{i},dx_{j})+2\prec \na_{X_{i}}df,X_{j}\succ.
   \end{equation*}
\end{proof}

We recall that a function $f\in C^{\infty}(M)$ is affine if $f$ is an affine function in local affine coordinates. In what follows, we introduce the following space.
\begin{equation*}
\mathcal{E}:=\{ f\in C^{\infty}(M)\ /\  \prec\na_{\al^{\#}}df,\be^{\#}\succ=0,\quad \forall  \al,\beta\in\Om^{1}(M)\};
\end{equation*}
which is the space of smooth functions $f:M\to \R$ that are affine along the leaves of the affine foliation (i.e. the restriction of $f$ to any leaf $L\subset M$ is affine).
\begin{example}
	Let $(\mathcal{A},.)$ be a $n$-dimensional commutative associative algebra. We have seen in \cite{ABB1} Proposition $4.2$ that the dual $\mathcal{A}^{*}$ carries a linear K-V structure $(\nabla,h)$ where $\nabla$ is the canonical affine connection on $\mathcal{A}^{*}$ and $h$ is the linear symmetric bivector field on $\mathcal{A}^{*}$ given by$$
	h(u,v)(\alpha)=\prec \alpha,u(\al).v(\al)\succ
	$$where $\al\in\mathcal{A}^{*}$ and $u,v\in\Om^{1}(\mathcal{A}^{*})=C^{\infty}(\mathcal{A}^{*},\mathcal{A})$. Now we take $\mathcal{A}=\mathbb{R}^{2}$ with the commutative associative product given by $e_{1}.e_{1}=e_{1}$ and the others products are zero and denote by $(x,y)$ the canonical dual coordinates of $\mathcal{A}^{*}$, so we have
	$$(dx)^{\#}=x\partial_{x}\text{ and }(dy)^{\#}=0.$$ 
	From Proposition \ref{Pro 0} it follows that for all $f\in C^{\infty}(\mathcal{A}^{*})$ which depend only on the $y$-variable $$\mathcal{L}_{X_{f}}(h)=0.$$
\end{example}

We know that for any K-V manifold $(M,\nabla,h)$ the tangent bundle $(TM,\Pi)$ is a Poisson manifold. So we can ask the following question: When the vertical lift $X_{f}^{v}$ and the horizontal lift $X_{f}^{h}$ are Poisson vector fields on $(TM,\Pi)$?  
\begin{proposition}{\label{Prop 0}}
	For any $f\in C^{\infty}(M)$ we have 
	\begin{enumerate}
		\item $X_{f}^{v}$ is the Hamiltonian vector field on $(TM,\Pi)$ associated to the function $f\circ p\in C^{\infty}(TM)$.
		\item $X_{f}^{h}$ is a Poisson vector field on $(TM,\Pi)$ if and only if  $f\in \mathcal{E}$.
	\end{enumerate} 
\end{proposition}
\begin{proof}
	 From (\ref{eq 11}) we get $X_{f}^{v}=\Pi_{\#}((df)^{h})=\Pi_{\#}(d(f\circ p))$ which leads to $1$. For the second assertion we need to compute $\mathcal{L}_{X_{f}^h}(\Pi)$. Let $\al,\be\in\Om^{1}(M)$, by using the formulas proved in \cite{ABB1} Proposition $3.6$, we get that  		
	\begin{equation*}
		\mathcal{L}_{X_{f}^h}(\Pi)(\alpha^{v},\beta^{v})=\mathcal{L}_{X_{f}^h}(\Pi)(\al^{h},\be^{h})=0.	\end{equation*}
Furthermore we have  
		\begin{eqnarray*}
		\mathcal{L}_{X_{f}^h}(\Pi)(\al^{v},\be^{h})&=&X_{f}^h.\Pi(\al^{v},\be^{h})-\Pi(\mathcal{L}_{X_{f}^h}\al^{v},\be^{h})-\Pi(\al^{v},\mathcal{L}_{X_{f}^h}\be^{h})\\
		&=& - X_{f}^h.(h(\al,\be)\circ p)-\Pi((\na_{X_{f}}\al)^v,\be^h)-\Pi(\al^v,(\mathcal{L}_{X_{f}}\be)^h)\\
		&=&-(X_{f}.h(\al,\be))\circ p + h(\na_{X_{f}}\al,\be)\circ p+h( \al,\mathcal{L}_{X_{f}}\be)\circ p\\
		&=&\left(-X_{f}.h(\al,\be)+ \prec\na_{X_{f}}\al,\be^\#\succ+\prec \mathcal{L}_{X_{f}} \be,\al^\#\succ \right) \circ p  \\
		&=&\left(\prec\mathcal{L}_{X_{f}} \be,\al^\#\succ-\prec\al,\na_{X_{f}}\be^\#\succ\right) \circ p  \\
		&=& \left( X_{f}.h(\al,\be)-\prec\be,[X_{f},\al^\#]\succ-\prec\al,\na_{X_{f}}\be^\#\succ  \right)\circ p \\
		&=& \left(\prec\na_{X_{f}}\al,\be^\#\succ-\prec\be,[X_{f},\al^\#]\succ\right)\circ p.
	\end{eqnarray*}
Now since $[\al^\#,\be^\#]=[\al,\be]_h^\#$ and $\na_{\al^\#}\be-\na_{\be^\#}\al=[\al,\be]_h$, we get 
\begin{equation*}
\mathcal{L}_{X_{f}^h}(\Pi)(\al^{v},\be^{h})=\prec \na_{\al^\#}df,\be^\#\succ\circ p.
\end{equation*}
\end{proof}

Now we consider the two vector spaces:
\begin{equation*}
\mathcal{V}_{h}:=\{X_{f}\in\Ga(TM)|\text{ } f\in \mathcal{E}\}\quad \text{ and }\quad\mathcal{V}_{\Pi}:=\{X_{f}^{h}\in\Ga(T(TM))|\text{ } f\in \mathcal{E}\}.
\end{equation*} 
\begin{proposition}{\label{Pro 1}}  $\mathcal{V}_{h}$ and $\mathcal{V}_{\Pi}$ are two abelian Lie subalgebras of vector fields.
\end{proposition}
\begin{proof} For all $f_1,f_2\in E$ we have
	\begin{equation*}
	[X_{f_1},X_{f_2}]=([df_1,df_2]_{h})^{\#}= \left( \na_{X_{f_1}}df_2-\na_{X_{f_2}}df_{f_1}\right)^{\#}=0.
	\end{equation*}
The other assumption follows from $[X^{h}_{f_1},X^{h}_{f_2}]=[X_{f_1},X_{f_2}]^{h}$. 	
\end{proof}

It is well known that for any affine manifold $(M,\na)$, the space of vector fields $\Ga(TM)$ endowed with the product $X\bullet Y=\na_{X}Y$ is a left symmetric algebra, i.e, for any $X,Y,Z\in\Ga(TM)$,
\[ \mathrm{ass}(X,Y,Z)= \mathrm{ass}(Y,X,Z),\]where $\mathrm{ass}(X,Y,Z)=(X\bullet Y)\bullet Z-X\bullet(Y\bullet Z)$.
It is therefore natural to look at the behavior of the subspaces $\mathcal{V}_{h}$ and $\mathcal{V}_{\Pi}$ with respect to the left symmetric structures in $\Ga(TM)$ and in $\Ga(TTM)$. 
\begin{theorem}\label{Theorem 0} Let $(M,\na,h)$ be a K-V manifold satisfying the following condition:
 \begin{equation}{\label{Special class}}
f_{1},f_{2}\in \mathcal{E}\Longrightarrow h(df_{1},df_{2}) \in \mathcal{E}. \end{equation}
Then  $(\mathcal{V}_{h},\bullet)$ and $(\mathcal{V}_{\Pi},\bullet)$ are two commutative associative subalgebras of $(\Ga(TM),\bullet)$ and $(\Ga(T(TM)),\bullet)$ respectively. Moreover, the map $X_{f}\mapsto X_{f}^{h}$ is an isomorphism from $\mathcal{V}_{h}$ to $\mathcal{V}_{\Pi}$.
\end{theorem}
\begin{proof} Let $f_{1},f_{2}\in \mathcal{E}$. For any  $\alpha\in\Om^{1}(M)$ we have 
	\begin{eqnarray*}
	\prec \al,X_{f_{1}}\bullet X_{f_{2}}\succ&=& \prec \al,(\mathcal{D}_{df_{1}}df_{2})^{\#}\succ\\
	&=& \prec \mathcal{D}_{df_{1}}df_{2},\al^{\#}\succ\\
	&=& \al^{\#}.h(df_{1},df_{2})\\
	&=& \prec \al,X_{h(df_{1},df_{2})}\succ.	
	\end{eqnarray*} 
Hence $X_{f_{1}}\bullet X_{f_{2}}=X_{h(df_{1},df_{2})}$ which is in $\mathcal{V}_{h}$ by the condition (\ref{Special class}). The remainder of the theorem is obvious.
\end{proof}
	The class of K-V manifolds satisfying condition (\ref{Special class}) deserves a special study that might be  the main subject of another paper; here  we only give an example of such class of K-V manifolds.
\begin{example} Let $(\mathcal{A},.,B)$ be a finite dimensional commutative and associative algebra endowed with a symmetric scalar $2$-cycle. We recall from \cite{ABB1} Proposition $4.2$ that the dual $\mathcal{A}^{*}$ carries a K-V structure $(\nabla,h)$ where $\nabla$ is the canonical affine structure on $\mathcal{A}^{*}$ and $h$ is the affine symmetric bivector field on $\mathcal{A}^{*}$ given by \begin{equation*}
		h(u,v)(\alpha)=\prec \alpha,u(\al).v(\al)\succ+B(u(\al),v(\al))
		\end{equation*}
		where $\al\in\mathcal{A}^{*}$ and $u,v\in\Om^{1}(\mathcal{A}^{*})$. We consider the canonical coordinates system $(x_{1},\ldots,x_{n})$ on $\mathcal{A}^{*}$ associated to a basis $(e_{1},\ldots,e_{n})$; then \begin{equation*}
		h(dx_{i},dx_{j})=b_{ij}+\sum_{k=1}^{n}C^{k}_{ij}x_{k},
		\end{equation*}
		where $C^{k}_{ij}=\prec e_{k}^{*},e_{i}.e_{j}\succ$ and $b_{ij}=B(e_{i},e_{j})$. Hence we get that a function $f$ belongs to $ \mathcal{E}$ if and only if for all $i,j=1,\cdots,n$ we have  \begin{equation}{\label{example 1}}
		\sum_{l,k}h_{il}h_{jk}\frac{\partial^{2}f}{\partial x_{l}\partial x_{k}}=0. 
		\end{equation} 
 Let $f_{1},f_{2}\in\mathcal{E}$, then we have
 $$h(df_{1},df_{2})=\displaystyle\sum_{i,j}h_{ij}\frac{\partial f_{1}}{\partial x_{i}}\frac{\partial f_{2}}{\partial x_{j}}.$$
 By a direct computation we get that:
	\begin{eqnarray*}
	\scalebox{0.8}[1]{$\displaystyle\sum_{k,l}h_{il}h_{jk}\frac{\partial^{2}h(df_{1},df_{2})}{\partial x_{k}\partial x_{l}}$}&=&\scalebox{0.8}[1]{$\displaystyle \sum_{k,l,m,s}h_{il}h_{jk}h_{ms}\frac{\partial^{3}f_{1}}{\partial x_{k}\partial x_{l}\partial x_{m}}\frac{\partial f_{2}}{\partial x_{s}}+h_{il}h_{jk}h_{ms}\frac{\partial^{2}f_{1}}{\partial x_{l}\partial x_{m}}\frac{\partial^{2} f_{2}}{\partial x_{k}\partial x_{s}}$}\\
	& &\scalebox{0.8}[1]{$\displaystyle+h_{il}h_{jk}C^{k}_{ms}\frac{\partial^{2}f_{1}}{\partial x_{l}\partial x_{m}}\frac{\partial f_{2}}{\partial x_{s}}+h_{il}h_{jk}h_{ms}\frac{\partial^{2}f_{1}}{\partial x_{k}\partial x_{m}}\frac{\partial^{2} f_{2}}{\partial x_{l}\partial x_{s}}$}\\
	& &\scalebox{0.8}[1]{$\displaystyle+h_{il}h_{jk}h_{ms}\frac{\partial f_{1}}{\partial x_{m}}\frac{\partial^{3}f_{2} }{\partial x_{k}\partial x_{l}\partial x_{s}}+h_{il}h_{jk}C^{k}_{ms}\frac{\partial f_{1}}{\partial x_{m}}\frac{\partial^{2} f_{2}}{\partial x_{l}\partial x_{s}}$}\\
	& &\scalebox{0.8}[1]{$\displaystyle+h_{il}h_{jk}C^{l}_{ms}\frac{\partial^{2} f_{1}}{\partial x_{k}\partial x_{m}}\frac{\partial f_{2}}{\partial x_{s}}+h_{il}h_{jk}C^{l}_{ms}\frac{\partial f_{1}}{\partial x_{m}}\frac{\partial^{2} f_{2}}{\partial x_{k}\partial x_{s}}.$}
	\end{eqnarray*}
Using the fact that
 $\displaystyle \sum_{k}h_{ik}C^{k}_{jm}=\sum_{k}h_{jk}C^{k}_{im}$ together with equation (\ref{example 1}) we get
 \begin{eqnarray*}
 	\scalebox{0.8}[1]{$\displaystyle\sum_{k,l,m,s}h_{il}h_{jk}h_{ms}\frac{\partial^{3}f_{1}}{\partial x_{k}\partial x_{l}\partial x_{m}}\frac{\partial f_{2}}{\partial x_{s}}$}&=&\scalebox{0.8}[1]{$\displaystyle \sum_{m,s}h_{ms}\frac{\partial f_{2}}{\partial x_{s}}\frac{\partial}{\partial x_{m}}\left(\sum_{k,l}h_{il}h_{jk}\frac{\partial^{2}f_{1}}{\partial x_{k}\partial x_{l}}\right)-\sum_{k,l,s}h_{jk}\frac{\partial^{2}f_{1}}{\partial x_{k}\partial x_{l}}\frac{\partial f_{2}}{\partial x_{s}}\left(\sum_{m}h_{ms}C^{m}_{il}\right)$}\\
 	& &\scalebox{0.8}[1]{$\displaystyle-\sum_{k,l,s}h_{il}\frac{\partial^{2}f_{1}}{\partial x_{k}\partial x_{l}}\frac{\partial f_{2}}{\partial x_{s}}\left(\sum_{m}h_{ms}C^{m}_{jk}\right)$}\\
 	&=&\scalebox{0.8}[1]{$\displaystyle-\sum_{m,s}C_{is}^{m}\frac{\partial f_{2}}{\partial x_{s}}\left(\sum_{k,l}h_{ml}h_{jk}\frac{\partial^{2}f_{1}}{\partial x_{k}\partial x_{l}}\right)-\sum_{m,s}C_{js}^{m}\frac{\partial f_{2}}{\partial x_{s}}\left(\sum_{k,l}h_{mk}h_{il}\frac{\partial^{2}f_{1}}{\partial x_{k}\partial x_{l}}\right)$}\\
 	&= &0
 \end{eqnarray*}
 and 
$$\displaystyle\sum_{k,l,m,s}h_{il}h_{jk}h_{ms}\frac{\partial^{2}f_{1}}{\partial x_{l}\partial x_{m}}\frac{\partial^{2} f_{2}}{\partial x_{k}\partial x_{s}}=\displaystyle \sum_{k,s}h_{jk}\frac{\partial^{2} f_{2}}{\partial x_{k}\partial x_{s}}\left(\sum_{l,m}h_{il}h_{sm}\frac{\partial^{2}f_{1}}{\partial x_{l}\partial x_{m}}\right)=0.$$
 Hence we get that $h(df_{1},df_{2})\in\mathcal{E}$. This implies that $(\mathcal{A}^{*},\nabla,h)$ satisfies condition (\ref{Special class}). 
\end{example}

In what follows, we give an example of a K-V manifold that does not satisfy the condition (\ref{Special class}).
\begin{example}
	 Consider the K-V manifold $M=(\mathbb{R}^{2},\na,h)$ where $\na$ is the canonical affine structure on $\mathbb{R}^{2}$ and $h=x^{2}\partial_{x}\otimes\partial_{x}$. Let $f:\mathbb{R}^{2}\rightarrow\mathbb{R},(x,y)\mapsto x$, clearly $f\in\mathcal{E}$ but $h(df,df)=f^{2}$ does not belong to the space $\mathcal{E}$ because  $$\prec\na_{(dx)^{\#}}df^{2},(dx)^{\#}\succ=f^{4}\neq 0.$$
\end{example}
\vskip 0.5cm
\section{Koszul-Vinberg maps}\label{sec2}

In this section, we will study smooth maps between Koszul-Vinberg manifolds that preserve these structures.

Let $(M^1,\nabla^1)$ and $(M^2,\nabla^2)$ be two affine manifolds. We will denote by $p_i:TM^i\rightarrow M^i$, the canonical projections. Let $F:M^{1}\rightarrow M^{2}$ be a smooth map. We recall that two vector fields $X\in\Ga(TM^{1})$ and $Y\in\Ga(TM^{2})$ are said to be $F$-related if and only if $T_xF(X_x)=Y_{F(x)}$ for all $x\in M^1$. The map $F$ is affine if and only if for any related vector fields $X^{1},X^{2}\in\Ga(TM^{1})$ and $Y^{1},Y^{2}\in\Ga(TM^{2})$ respectively, the two vector fields $\nabla^{1}_{X^1}{X^2}\in\Ga(TM^{1})$ and $\nabla^{2}_{Y^1}{Y^2}\in\Ga(TM^{2})$ are $F$-related; in a more geometric terms: $F$ is affine if and only if for every vector parallel vector field $X$ along a curve $\gamma$ on $M^{1}$ the image $F_{*}X$ is also parallel along the curve $F\circ\gamma$, or equivalently $F$ is totally geodesic; i.e. the image $F\circ\gamma$ of each geodesic $\gamma$ of $M^{1}$ is a geodesic of $M^{2}$ (see \cite{Kobayashi1}).

Let $(M^{i},\na^{i},h^{i})$ for $i=1,2$ be two K-V manifolds and denote by $\Pi^{i}$ the induced Poisson tensor fields on $TM^{i}$.
\begin{theorem}\label{Theorem 1} 
	Let $F:(M^{1},\nabla^{1},h^{1})\rightarrow(M^{2},\nabla^{2},h^{2})$ be an affine map. Then  the following assertions are equivalent:
	\begin{enumerate}
		\item[$(i)$] $F$ is a K-V map, i.e. for any $\al,\be\in\Om^{1}(M^{2})$ we have
		\begin{equation}{\label{K-V map}}
		h^{1}(F^{*}\al,F^{*}\be)=h^{2}(\al,\be)\circ F.
		\end{equation} 
		\item[$(ii)$] The tangent map $TF:(TM^{1},\Pi^{1})\rightarrow(TM^{2},\Pi^{2})$ is a Poisson map.
		\item[$(iii)$] For any one form $\alpha\in\Omega^{1}(M^{2})$ the vector fields $(F^{*}\alpha)^{\#_{1}}\in\Gamma(TM^{1})$ and $\alpha^{\#_{2}}\in\Gamma(TM^{2})$ are $F$-related.
		\item[$(iv)$] For any $f\in C^{\infty}(M^{2})$, the K-V Hamiltonian vector fields $X_{f}\in\Gamma(TM^{2})$ and $X_{f\circ F}\in\Gamma(TM^{1})$ are $F$-related.
	\end{enumerate}
\end{theorem}

In order to show this theorem, we need the following lemma.  
\begin{lemma}\label{Lemma1}
	Let $F:(M^1,\nabla^1)\rightarrow(M^2,\nabla^2)$ be an affine map. Then we have 
	\begin{enumerate}
		\item For any $F$-related vector fields $X\in\Gamma(TM^1)$ and $Y\in\Gamma(TM^2)$, their respective vertical lifts $X^{v}$ and $Y^{v}$ are $TF$-related. The same thing happens for their horizontal lifts $X^{h}$ and $Y^{h}$.
		\item For any $\alpha\in\Omega^{1}(M^2)$ we have
		\begin{equation*}
		(F^{*}\alpha)^{v}=(TF)^{*}(\alpha^{v}) \quad \text{and} \quad (F^{*}\alpha)^{h}=(TF)^{*}(\alpha^{h}).
		\end{equation*} 
	\end{enumerate}	
\end{lemma}
\begin{proof} 
	Let $x\in M^1$ and $u\in T_{x}M^1$, then we have  
	\begin{eqnarray*}
		T_{u}(TF)(X^v_u)&=& \frac{d}{dt}_{|_{t=0}} T_xF(u+tX_x)\\
		&=& \frac{d}{dt}_{|_{t=0}} T_{x}F(u)+tT_{x}F(X_x)\\
		&=& \frac{d}{dt}_{|_{t=0}} T_{x}F(u)+tY_{F(x)}\\
		&=& Y^v_{TF(u)}.
	\end{eqnarray*}
	Hence $X^v$ and $Y^v$ are $TF$-related. Now let $\gamma:I\rightarrow M^{1}$ be a curve with $\gamma(0)=x$ and $\gamma^\prime(0)=X_x$. Now using equation (\ref{eq 6}) we get
	\begin{equation*}
	T_{u}(TF)(X^h_u)= \frac{d}{dt}_{|_{t=0}} T_{\gamma(t)}F(\tau^\gamma_{0t}(u) ),
	\end{equation*} 
	and by the commutativity of the parallel transport maps, we obtain
	\begin{equation*}
	T_{u}(TF)(X^h_u)=  \frac{d}{dt}_{|_{t=0}} \tau^{F\circ \gamma}_{0t}(TF(u))=Y^h_{TF(u)}.
	\end{equation*} 
	Hence $1)$. Now from this first step, we deduce that for any $X\in\Gamma(TM^{1})$ and $\alpha\in\Omega^{1}(M^{2})$, we have
	\begin{eqnarray*}
		\prec(TF)^{*}\alpha^{v},X^v\succ&=& \prec \alpha^{v},Y^v\succ\circ TF\\
		&=& \prec \alpha,Y\succ\circ p_2\circ  TF\\
		&=& \prec\alpha,Y\succ\circ F\circ  p_1\\
		&=& \prec F^*\alpha,X\succ\circ p_1\\
		&=& \prec (F^*(\alpha))^v,X^v\succ.
	\end{eqnarray*}
	Furthermore,
	\begin{equation*}
	\prec(TF)^{*}(\alpha^{v}),X^h\succ=\prec\alpha^{v},Y^h\succ\circ TF=0=\prec(F^*(\alpha))^v,X^h\succ.
	\end{equation*} 
	This proves that $(TF)^{*}(\alpha^{v})=(F^*(\alpha))^v$. Similarly we get $(TF)^{*}(\alpha^{h})=(F^*(\alpha))^h$. 
\end{proof} 
\begin{proof}[Proof of Theorem \ref{Theorem 1}]
	$(i)\Longrightarrow(ii)$ According to the Lemma \ref{Lemma1} we get that for any $\alpha,\beta\in\Omega^{1}(M^{2})$
	\begin{eqnarray*}
		\Pi^{1}((TF)^{*}(\alpha^h),(TF)^{*}(\beta^h))&=&\Pi^{1}((F^{*}\alpha)^{h},(F^{*}\beta)^{h})\\
		&=& 0 \\
		&=& \Pi^{2}(\alpha^{h},\beta^{h})\circ TF.
	\end{eqnarray*}
	Similarly we get  $\Pi^{1}((TF)^{*}(\alpha^v),(TF)^{*}(\beta^v))=\Pi^{2}(\alpha^{v},\beta^{v})\circ TF$. Furthermore, we have
	\begin{eqnarray*}
		\Pi^{1}((TF)^{*}(\alpha^h),(TF)^{*}(\beta^v))&=&\Pi^{1}((F^{*}\alpha)^{h},(F^{*}\beta)^{v})\\
		&=& h^{1}(F^{*}\alpha,F^{*}\beta)\circ p_{1} \\
		&=& h^{2}(\alpha,\beta)\circ F\circ p_{1}\\
		&=& h^{2}(\alpha,\beta)\circ p_{2}\circ TF\\
		&=& \Pi^{2}(\alpha^{h},\beta^{v})\circ TF.
	\end{eqnarray*}
	
	$(ii)\Longrightarrow(i)$ Since $TF$ is Poisson map hence for any one forms $\alpha,\beta\in\Omega^{1}(M^{2})$ we have 
	\begin{equation*}
	\Pi^{2}(\alpha^{h},\beta^{v})\circ TF=\Pi^{1}((TF)^{*}(\alpha^h),(TF)^{*}(\beta^v)).
	\end{equation*}
	Hence according to the Lemma \ref{Lemma1} we get that 
	\begin{equation*}
	h^{2}(\alpha,\beta)\circ p_{2}\circ TF=h^{1}(F^{*}\alpha,F^{*}\beta)\circ p_{1}.
	\end{equation*}
	Since $p_{2}\circ TF=F\circ p_{1}$ 
	\begin{equation*}
	h^{2}(\alpha,\beta)\circ F\circ p_{1}=h^{1}(F^{*}\alpha,F^{*}\beta)\circ p_{1},
	\end{equation*}
	and then from the surjectivity of $p_1$, the equality (\ref{K-V map}) follows.
	
	$(iii)\Leftrightarrow(i)$ Let $\alpha\in\Omega^{1}(M^{2})$ and put $X:=(F^{*}\alpha)^{\#_{1}}\in\Gamma(TM^{1})$ and $Y:=\alpha^{\#_{2}}\in\Gamma(TM^{2})$. Then equation (\ref{K-V map}) becomes $(F^*\beta)(X)=\beta(Y)\circ F$ for any differential $1$-form $\beta\in\Omega^{1}(M^{2})$; which is equivalent to the fact that the vector fields $X$ and $Y$ are $F$-related.
	
	$(iii)\Leftrightarrow(iv)$ Is obvious.	
\end{proof}

\begin{example}$\ $ \label{Example1}
	\begin{enumerate}
		\item Let $(M^{i},\nabla^{i},h^{i})$ for $i=1,2$ be two K-V manifolds, and consider the K-V tensor field $h=h^1\oplus h^2$ on the affine product manifold $M^{1}\times M^{2}$ (as described in \cite{ABB1} Proposition $2.9$). Then the canonical projections $$p_{i}: (M^{1}\times M^{2},\nabla,h)\rightarrow (M^{i},\nabla^{i},h^{i})$$ are a K-V maps.   
		\item Let $(G,\na,h)$ be a simply connected Lie group equipped with a K-V structure. Then the multiplication map $$m:(G\times G,\na\oplus\na,h\oplus h)\too (G,\na,h)$$ is a K-V map if and only if $G$ is a vector space, $\na$ its canonical affine connection and $h$ is linear (see Corollary $4.5$ in \cite{ABB1}).
		\item Let $(\mathcal{A}_{i},\bullet,B_{i})$ for $i=1,2$ be two commutative associative algebras endowed with two symmetric scalar
		2-cocycle and denote by $(\mathcal{A}^{*},\na,h^{i})$ its associated K-V manifolds. Consider an affine map $F:\mathcal{A}_{1}\rightarrow\mathcal{A}_{2}$, 
		which satisfy  $$F(u\bullet v)=F(u)\bullet F(v)\text{ and } B_{1}(u,v)=B_{2}(f(u),f(v))$$ for all $u,v\in\mathcal{A}_{1}$. Then its dual map $$F^{*}:(\mathcal{A}_{2}^{*},\na,h^{2})\rightarrow (\mathcal{A}_{1}^{*},\na,h^{1})$$ is a K-V map.
		\item  Let $(\mathfrak{g}_{1},\bullet,r^{1},r^{2})$ a K-V algebra (see Definition $6.2$ in \cite{ABB1}) and $(\mathfrak{g}_{1},\bullet)$ a left symmetric algebra. Denote by $(G_{1},\na^{1},r^{1-})$ the K-V manifold associated to $(\mathfrak{g}_{1},\bullet,r^{1})$ and by $(G_{2},\na^{2})$ the affine manifold associated to $(\mathfrak{g}_{1},\bullet)$. Let $F:\mathfrak{g}_{1}\rightarrow\mathfrak{g}_{2}$ be a morphism of left symmetric algebra, $r^{2}$ be the symmetric bivector on $\mathfrak{g}_{2}$ given by:
		$$r^{2}(\al,\be):=r^{1}(F^{*}\al,F^{*}\be),\,\al,\be\in\mathfrak{g}^{*}_{2}$$ 
		and denote by $\tilde{F}:G_{1}\rightarrow G_{2}$ the map which integrate $F$. Then $(\mathfrak{g}_{2},\bullet,r^{2})$ is K-V algebra and  $$\tilde{F}:(G_{1},\na^{1},r^{1-})\rightarrow (G_{2},\na^{2},r^{2-})$$
		is a K-V map.	
	\end{enumerate}
\end{example}

As a consequence of $1.$ in Example \ref{Example1} and Theorem \ref{Theorem 1} we obtain:
\begin{corollary}\label{Corollary5}
	The canonical diffeomorphism $$\psi:=Tp_{1}\times Tp_{2}:(T(M^{1}\times M^{2}),\Pi)\rightarrow (TM^{1}\times TM^{2},\Pi^{1}\oplus\Pi^{2}),$$ is a Poisson map, where $\Pi$ is the Poisson tensor associated to the K-V tensor field $h=h^1\oplus h^2$.
\end{corollary}
\begin{proof}
	We know that skew-symmetric tensor $\Pi^{1}\oplus\Pi^{2}$ is the unique Poisson structure on the product manifold $TM^{1}\times TM^{2}$ such that the canonical projections $\widetilde{p}_{i}:(TM^{1}\times TM^{2},\Pi^{1}\oplus\Pi^{2})\rightarrow (TM^{i},\Pi^{i})$ are Poisson maps. Since the one forms given by $\widetilde{p}_{i}^{*}\alpha^{h}_{i}$ and $\widetilde{p}_{i}^{*}\alpha^{v}_{i}$ where $\alpha_{i}\in\Gamma(T^{*}M^{i})$ span the vector space of one forms on $TM^{1}\times TM^{2}$. Then for any $\alpha_{1},\beta_{1}\in\Gamma(T^{*}M^{1})$ we have 
	\begin{eqnarray*}
		\Pi^{1}\oplus\Pi^{2}\left(\widetilde{p}_{1}^{*}\alpha^{h}_{1},\widetilde{p}_{1}^{*}\beta^{v}_{1}\right)\circ \psi&=& \Pi^{1}(\alpha^{h}_{1},\beta^{v}_{1})\circ \widetilde{p}_{1}\circ \psi\\
		&=&	\Pi^{1}(\alpha^{h}_{1},\beta^{v}_{1})\circ Tp_{1}
	\end{eqnarray*}
	According to Theorem \ref{Theorem 1} it follow that 
	$Tp_{i}:(T(M^{1}\times M^{2}),\Pi)\rightarrow (TM^{i},\Pi^{i})$ are Poisson maps. Therefore we have 
	\begin{eqnarray*}
		\Pi^{1}\oplus\Pi^{2}\left(\widetilde{p}_{1}^{*}\alpha^{h}_{1},\widetilde{p}_{1}^{*}\beta^{v}_{1}\right)\circ\psi&=& \Pi\left(Tp_{1}^{*}\alpha^{h}_{1},Tp_{1}^{*}\beta^{h}_{1}\right)\\	
		&=& \Pi\left(\psi^{*}(\widetilde{p}_{1}^{*}\alpha^{h}_{1}),\psi^{*}(\widetilde{p}_{1}^{*}\beta^{v}_{1})\right).
	\end{eqnarray*}
	Similarly, we verify that we have the same last equality for all types of one forms that span the vector space of one forms on $TM^{1}\times TM^{2}$. This implies that $\psi$ is a Poisson map. 
\end{proof}

We finish this section by giving some proprieties about K-V maps.
\begin{proposition}  
	Let $(M^{i},\nabla^{i},h^{i})$ for $i = 1, 2, 3$ be a K-V manifolds, $\phi:(M^{1},\nabla^{1},h^{1})\rightarrow(M^{2},\nabla^{2},h^{2})$ a K-V map and
	$\psi:(M^{2},\nabla^{2},h^{2})\rightarrow(M^{3},\nabla^{3},h^{3})$ an affine map.
	\begin{enumerate}
		\item For any $x\in M^1$, we have $
		\mathrm{rank}(h^{1}_{\#}(x))\geq\mathrm{rank}(h^2_{\#}(\phi(x)))$.
		\item If $(M^{2},\nabla^{2},h^{2})$ is a pseudo-Hessian manifold and $\phi$ is surjective, then $\phi$ is a submersion map.
		\item If $\psi$ is a K-V map, then $\psi\circ\phi$ is a K-V map.
		\item If $\psi\circ\phi$ is a K-V map, and $\phi$ is surjective, then $\psi$ is a K-V map.
		\item If $\phi$ is a diffeomorphism, then $\phi^{-1}$ is a K-V map.
	\end{enumerate}
\end{proposition}

\vskip 0.5cm
\section{Koszul-Vinberg submanifolds}\label{sec3}

By similarity with Poisson submanifolds studied in (\cite{Fernandes},\cite{Laurent-Gengoux},\cite{Weinstein}), we will introduce in this section a class of submanifolds in the context of K-V manifolds.
\vskip 0.15cm
Let $(M,\nabla,h)$ be a K-V manifold. We say that  an immersed submanifold $\iota:N \hookrightarrow M$ is a {\textit{Koszul-Vinberg submanifold}} of $(M,\nabla,h)$, if it can be equipped with a connection $\na^N$ and a K-V bivector field $h^N$ such that the immersion $\iota:N\hookrightarrow M$ becomes a K-V map. In particular, $(N,\na^N)$ is an affine submanifold as studied in (\cite{Pawel}). 

For any vector subspace $V$ of a finite-dimensional vector space $E$ we denote by $V^{\circ}$ the annihilator of $V$ in $E^{*}$.
\begin{proposition}\label{Propostion1}
	Given an immersed affine submanifold $\iota:(N,\nabla^{N})\hookrightarrow(M,\nabla)$, then there is at most one K-V bivector field $h^N$ on $N$ that makes $(N,\nabla,h^N)$ into a K-V submanifold. This happens if and only if any of the following equivalent conditions hold:
	\begin{enumerate}
		\item $\mathrm{Im}(h_{\#}(\iota(x)))\subset T_{x}\iota(T_{x}N)$ for all $x\in N$;
		\item For any $f\in C^\infty(M)$, the K-V Hamiltonian vector field $X_{f}$ is tangent to $\iota(N)$;
		\item For all $x\in N$, $h_{\#}((T_x\iota(T_xN))^{\circ})=0$.
	\end{enumerate}
	If $N$ is a closed submanifold, then these conditions are also equivalent to:
	\begin{itemize}
		\item[$4.$] For any $f\in C^{\infty}(M)$ and $g\in\mathfrak{I}(N):=\{g\in C^{\infty}(M)|\text{ } g_{\mid_N}=0\}$ we have $h(df,dg)\in\mathfrak{I}(N)$. 
	\end{itemize}
\end{proposition}
\begin{proof}
	If $\iota:(N,\nabla^N,h^N)\hookrightarrow(M,\nabla,h)$ is a K-V map, then the two vector fields $h$ and $h^N$ are $\iota$-related, this means that for all $x\in N$ 
	\begin{equation*}
	T_{x}\iota\circ h^N_{\#}(x)\circ (T_{x}\iota)^{*}=h_{\#}(\iota(x)).
	\end{equation*}
	Since $T_{x}\iota$ is injective, this shows that $h^N$ is unique. It also shows that $1.$ should be if $(N,\nabla^N,h^N)$ is a K-V submanifold.
	
	Next suppose that $\iota:(N,\nabla^N)\hookrightarrow(M,\nabla)$ is an affine immersion that satisfies $\mathrm{Im}(h_{\#}(i(x)))\subset T_{x}(i)(T_{x}N)$. We claim that there exists a unique smooth bivector field $h^N$ in $N$ such that $h_{\#}(\iota(x))$ factors as:
	\begin{equation*}
	\begin{tikzcd}[sep=huge]
	T^{*}_{\iota(x)}M \arrow[r,"h_{\#}(\iota(x))"] \arrow[d, "(T_{x}\iota)^{*}"']& T_{\iota(x)}M \\ T^{*}_{x}N \arrow[r, "h^N_{\#}(x)"',dashrightarrow]& T_{x}N \arrow[u, "T_{x}\iota"'] \end{tikzcd}
	\end{equation*}
	Since we already know that  $\mathrm{Im}(h_{\#}(\iota(x)))\subset T_{x}(\iota)(T_{x}N)$, it is enough to check that for any $\alpha\in (T_{x}\iota(T_{x}N))^{\circ}$ we have $\alpha^\#=0$. In fact, we find for any $\beta\in T^{*}_{\iota(x)}M$:
	\begin{equation*}
	\langle\beta,\alpha^\#\rangle=\langle\alpha,\beta^\#\rangle=0,
	\end{equation*}
	which proves the claim (the smoothness of $h^N$ is automatic).
	
	Now observe that the skew-symmmetric bivector field $\Pi^N$ on $TN$ associated to $(\nabla^N,h^N)$ and the Poisson bivector field $\Pi$ on $TM$ associated to $(\nabla,h)$ are $T\iota$-related, this implies that the Schouten brackets $[\Pi^N,\Pi^N]$ and $[\Pi,\Pi]$ are also $T\iota$-related hence we have 
	\begin{equation*}
	[\Pi^N,\Pi^N]=0.
	\end{equation*}
	According to Theorem $3.3$ in (\cite{ABB1}) $(N,\nabla^N,h^N)$ is a K-V manifold. This shows that if $1.$ holds, then $N$ has a unique K-V structure $(\nabla^N,h^N)$ such that the immersion $\iota:(N,\nabla^N,h^N)\hookrightarrow(M,\nabla,h)$ is a K-V map.
	
	The equivalence $1. \Leftrightarrow 2.$ follows from the fact that the vector space $\mathrm{Im}(h_{\#}(\iota(x)))=\mathrm{span}\{h_{\#}(\iota(x))(df)| f\in C^{\infty}(M)\}$.
	
	The equivalence $2.\Leftrightarrow 3.$ follows from observing that for any  $\alpha\in (T_{x}\iota(T_{x}N))^{\circ}$ and $\beta\in T^{*}_{\iota(x)}M$ we have $\langle\beta,\alpha^\#\rangle=\langle\alpha,\beta^\#\rangle,
$ ;	so $h_{\#}((T_{x}\iota(T_{x}N))^{\circ})=0$
	if and only if $h_{\#}(\iota(x))(T_{\iota(x)}^{*}M)\subset T_{x}\iota(T_{x}N)$.
	Finally, notice that if $N$ is a closed submanifold, a vector field $X\in\Gamma(TM)$ is tangent to $N$ if and only if for any $f\in\mathfrak{I}(N)$ we have $X(f)\in\mathfrak{I}(N)$. Hence, the result follows from the first part and the fact that $h(df,dg)=X_{g}(f)$.
\end{proof}

Now let $(N,\na)$ be an affine submanifold $(M,\nabla,h)$ and $h^{N}$ be a symmetric bivector field on $N$.
\begin{corollary}\label{Corollary1}
	The following statement are equivalent. 
	\begin{enumerate}
		\item $(N,\nabla^N,h^N)$ is a K-V submanifold.
		\item $(TN, \Pi^N)$ is a Poisson submanifold of $(TM,\Pi)$.
	\end{enumerate}
\end{corollary}
\begin{example}\label{Example2} $\ $ 
	\begin{enumerate}
		\item Take $M=(\mathbb{R}^{2},\nabla,h^{2})$ endowed with its canonical affine structure and the K-V bivector $h^2=x^{2}\partial_{x}\otimes\partial_{x}+y^{2}\partial_{y}\otimes\partial_{y}$. We consider the following affine immersion $$F:\mathbb{R}\rightarrow\mathbb{R}^{2},x\mapsto (\lambda x,\mu x), (\lambda,\mu)\neq(0,0)$$ where $\mathbb{R}$ is endowed with its canonical affine structure. Then the only way to make $F$ as a K-V map is to take $\lambda=0$ or $\mu=0$ and  $h^{1}=x^{2}\partial_{x}\otimes\partial_{x}$.
		\item Take $M=(\mathbb{R}^{m},\nabla,h^{m})$ and  $N=(\mathbb{R}^{m-k}\times\{0_{k}\},\nabla,h^{m-k})$ endowed with its canonical affine structures and the  K-V bivectors $$h^m=\sum_{i,j=1}^{m}x_{i}x_{j}\partial_{x_{i}}\otimes\partial_{x_{j}}\text{ and }  h^{m-k}=\sum_{i,j=1}^{m-k}x_{i}x_{j}\partial_{x_{i}}\otimes\partial_{x_{j}}.$$ Then $N$ is a K-V submanifold of $M$.
		\item Consider $(\mathbb{R}^n,\na,h)$ endowed with its canonical affine connection and 
		\begin{equation*}
		h=\sum_{i,i=1}^{m}f_{i}(x_{i})\partial_{x_{i}}\otimes\partial_{x_{i}}
		\end{equation*}
		where $f_{i}\in C^{\infty}(\mathbb{R})$. Then $N=\mathbb{R}^{m-k}\times\{0_{k}\}$ as an affine submanifold $(\mathbb{R}^n,\na,h)$ is a K-V submanifold if and only if $f_{i}(0)=0$, for $i=k+1,\ldots,n$.
		\item Let $(M,\na,h)$ be a K-V manifold and $f\in C^{\infty}(M)$ an affine function, i.e. $\na df=0$, such that $X_{f}(g)=0$ for all $g\in C^{\infty}(M)$. Then all the smooth level sets of $f$ are K-V submanifolds. Indeed, since $X_{g}(f)=X_{f}(g)=0$ shows that all K-V Hamiltonian vector fields are tangent to the level sets of $f$. 
		\item Let $(\mathcal{A}^{*},\na,h)$ be the linear K-V manifold defined in \cite{ABB1} and $I\subset (\mathcal{A},.)$ be an ideal. Then $I^{\circ}\subset \mathcal{A}^{*}$ is a K-V submanifold.
		\item Let $(\mathfrak{g},.,r)$ be K-V algebra and $\mathfrak{h}\subset(\mathfrak{g},.)$ be a left symmetric subalgebra such that $\mathrm{Im}(r_{\#})\subset\mathfrak{h}$. Denote by $(G,\na,h)$ the associated K-V manifold to $(\mathfrak{g},.,r)$ and by $(H,\na^{H})$ be the associated affine Lie subgroup to $(\mathfrak{h},.)$. Then $H$ is a K-V submanifold.  
	\end{enumerate}
\end{example}

In the following, we show that a K-V structure can be defined by its affine foliation instead of the K-V bivector.
\begin{theorem}
	Let $(M,\na)$ be an affine manifold, and $\mathcal{F}$ a general foliation such that:
	\begin{enumerate}
		\item Every leaf $L$ of $\mathcal{F}$ is endowed with a pseudo-Hessian structure $(\nabla^{L},g_{L})$ and $(L,\na^{L})$ is an affine submanifold of $(M,\na)$.
		\item If $f\in C^{\infty}(M)$, the vector field $X_{f}$ defined by $X_{f}(x)=$ the gradient vector field of $f|_{L}$ on $(L,\na^{L},g_{L})$ at $x$ is a smooth vector field on $M$ where $L$ is the leaf passing through $x$.
	\end{enumerate}
	Then $(M,\na)$ has unique K-V bivector $h$ whose affine foliation is $\mathcal{F}$. Moreover each leaf $(L,\na^{L},h^{L})$ is a K-V submanifold of $(M,\na,h)$ where $h^{L}=g_{L}^{-1}$.
\end{theorem} 
\begin{proof} We define a symmetric bivector on $M$ by putting for all $\al,\be\in\Om^{1}(M)$ and $x\in M$ 
	\begin{equation*}
	h(\al,\be)(\iota(x))=h^{L}(\iota^{*}\al,\iota^{*}\be)(x)
	\end{equation*}
	where $L$ is the affine leaf passing through $x\in M$, the map $\iota:L\hookrightarrow M$ is the canonical injection. The smoothness of $h$ follow automatically from $2$. From $1$ and $2$ one can deduce that for any one forms $\al,\be,\ga\in\Om^{1}(M)$ we have 
	\begin{eqnarray*}
		\na_{\al^{\#}}(h)\left(\be,\ga\right)(\iota(x))
		&=& \na^{L}_{(\iota^{*}\al)^{\#_{L}}}(h^{L})\left(\iota^{*}\be,\iota^{*}\ga\right)(x)\\
		&=& \na^{L}_{(\iota^{*}\be)^{\#_{L}}}(h^{L})\left(\iota^{*}\al,\iota^{*}\ga\right)(x)\\
		&=& \na_{\be^{\#}}(h)\left(\al,\ga\right)(\iota(x)).
	\end{eqnarray*}
\end{proof}

Now we look at the relation between the notion of K-V submanifolds and the affine foliation.
\begin{proposition}
	Let $(M,\na,h)$ be a K-V manifold with affine foliation $\mathcal{F}$. An affine submanifold $(N,\na^{N})\subset (M,\na)$ is a K-V submanifold if and only if for each leaf $L\in\mathcal{L}$ the intersection $L\cap N$ is an open subset of $L$. Hence, the affine foliation of $(N,\na^{N},h^{N})$ consists of the connected components of the intersection $L\cap N$.
\end{proposition}
\begin{proof}
	An affine submanifold $(N,\na^{N})\subset (M,\na)$ is a K-V submanifold if and only if 
	\begin{equation*}
	\mathrm{Im}(h_{\#}(x))\subset T_{x}N,\text{ }\forall x\in N.
	\end{equation*}
	It follows that for a K-V submanifold $N\subset M$, every affine leaf of $(N,\na^{N},h^{N})$ is also an integral submanifold of $(M,\na,h)$. Hence, every affine leaf of $(N,\na^{N},h^{N})$ is an open subset of an affine leaf of $\mathcal{F}$.
	
	Conversely, if for each affine leaf $L\in \mathcal{F}$ the intersection $L\cap N$ is an open subset of $L$. Then for any $x\in N$ we have $\mathrm{Im}(h_{\#}(x))=T_{x}L\subset T_{x}N$, where $L\in \mathcal{F}$ is the affine leaf passing through $x$. This shows that $N$ is a K-V submanifold. 
\end{proof}
\begin{remark}
	Let $(M,\nabla,g)$ be a pseudo-Hessian manifold. Then the only K-V submanifolds are the open subsets of $M$.
\end{remark}

Finally, we give the relation between K-V maps and the affine foliation.
\begin{proposition}
	If $F:(M^{1},\na^{1},h^{1})\rightarrow (M^{2},\na^{2},h^{2})$ is K-V map then for each pseudo-Hessian leaf $L$ of $(M^{2},\na^{2},h^{2})$, the set $L\cap\mathrm{Im}(F)$ is open in $L$. In particular, if $\mathrm{Im}(F)$ is an affine submanifold, then it is a K-V submanifold of $(M^{2},\na^{2},h^{2})$. 
\end{proposition}
\begin{proof}
	Let $x\in M^{1}$ and set $y=F(x)$. Denote by $L$ the affine leaf of $(M^{2},\na^{2},h^{2})$ containing $y$. Any point in $L$ can be reached from $y$ by piecewise smooth curves consisting of integral curves of K-V Hamiltonian vector field $X_{f}$.
	
	Given $f\in C^{\infty}(M)$, by Theorem \ref{Theorem 1}, the vector fields $X_{f}$ and $X_{f\circ F}$ are $F$-related. Hence, if $\ga_{2}(t)\in M^{2}$ and $\ga_{1}(t)\in M^{1}$ are the integral curves of $X_{f}$ and $X_{f\circ F}$ satisfying $\ga_{2}(0)=y$ and $\ga_{1}(0)=x$, we have: $\ga_{2}(t)=F(\ga_{1}(t))$, for all small enough $t$. It follows that a neighborhood of $y$ in $L$ is contained in the image of $F$.
	
\end{proof}
\begin{corollary}\label{Corollary3}
	Let $(M,\nabla,h)$ be a K-V manifold. If $N^{1},N^{2}\subset M$ are two K-V submanifolds which intersect transversely then $N^{1}\cap N^{2}\subset M$ is also a K-V submanifold.
\end{corollary}
\begin{proof} 
	Let $\gamma:I\rightarrow N^{1}\cap N^{2}$ be a curve with $\gamma(0)=x$ and $\gamma(1)=y$. We take $u\in T_{x}(N^{1}\cap N^{2})=T_{x}N^{1}\cap T_{x}N^{2}$. Since $N^{1}$ and $N^{2}$ are affine submanifolds of $M$ then $\tau^{\gamma}(u)\in T_{y}N^{1}\cap T_{y}N^{2}=T_{y}(N^{1}\cap N^{2})$ where $\tau^{\gamma}:T_{x}M\rightarrow T_{y}M$ is the parallel transport along $\gamma$. This show that $N^{1}\cap N^{2}$ is an affine submanifold of $M$. Applying the first assertion of the Proposition \ref{Propostion1} to $N^{1}$ and $N^{2}$, we get that for all $x\in N^{1}\cap N^{2}$, $\mathrm{Im}(h_{\#}(x))\subset T_{x}N^{1}\cap T_{x}N^{2}=T_{x}(N^{1}\cap N^{2})$. Hence $N^{1}\cap N^{2}$ is a K-V submanifold of $(M,\nabla,h)$.
\end{proof}

In general, K-V submanifolds don't have functoriality under K-V maps which are explained by the following example.
\begin{example}\label{Example4}
	Consider the K-V map
	\begin{eqnarray*}
		&& F:(\mathbb{R}^{3},\nabla,h^{1})\rightarrow(\mathbb{R}^{2},\nabla,h^{2})\\
		& & \qquad (x,y,z)\longmapsto (x+y-\sqrt{2}z,x+y-\sqrt{2}z)
	\end{eqnarray*}	
	where $\nabla$ is the canonical affine structure  of $\mathbb{R}^{n}$,  $$h^{1}=\partial_{x}\otimes\partial_{x}+\partial_{y}\otimes\partial_{y}-\partial_{z}\otimes\partial_{z}\text{ and }h^{2}=0.$$The affine submanifold $N:=\{(x,-x)|\text{ } x\in\mathbb{R}\}\subset\mathbb{R}^{2}$ is a transversal to the map $F$, moreover we have  $F^{-1}(N)=\{(x,y,z)|\text{ }x+y-\sqrt{2}z=0\}\subset\mathbb{R}^{3}$ is an hyperplane hence  $F^{-1}(N)$ cannot be a K-V submanifold of $\mathbb{R}^{3}$, since the only one are the open subset of $\mathbb{R}^{3}$.
\end{example}

\section{Koszul-Vinberg transversals}\label{sec4}

We recall from (\cite{Fernandes}) that a Poisson transversal of a Poisson manifold $(P,\pi)$ is a submanifold $N\subset P$ such that, at every point $x\in N$, we have 
\begin{equation*}
T_{x}P=T_{x}N+\pi_{\#}(T_{x}N^{\circ}).
\end{equation*} 

\begin{definition}\label{definition1}
	A {\textit{Koszul-Vinberg transversal}} of a K-V manifold $(M,\nabla,h)$ is an affine submanifold $N\subset M$ such that, at every point $x\in N$, we have 
	\begin{equation}\dreqno\label{eq7}
	T_{x}M=T_{x}N+h_{\#}(T_{x}N^{\circ}).
	\end{equation} 
\end{definition}

Note that the equality $\mathrm{rank}(TN^{\circ})=\mathrm{rank}(T_{N}M)-\mathrm{rank}(TN)$ implies that condition (\ref{eq7}) is equivalent to the direct sum decomposition:
\begin{equation}\label{eq8}
T_{N}M=TN\oplus h_{\#}(TN^{\circ}).
\end{equation}

The main reason to consider K-V transversals is that they have naturally induced K-V structures. Indeed, let $N$ be a K-V tranversal of $(M,\nabla,h)$. Then the decomposition for $T_{N}M$ and the dual decomposition for $T^{*}_{N}M$ gives a sequence of bundle maps:
\begin{equation*}
T^{*}N\stackrel{p^{*}}{\longrightarrow} T^{*}_{N}M\stackrel{h_{\#}}{\longrightarrow}T_{N}M\stackrel{p}{\longrightarrow}TN.
\end{equation*}
The resulting  bundle map $T^{*}N\rightarrow TN$ is symmetric and so it is of the form $h^{N}_{\#}$ for a unique symmetric bivector field $h^{N}$ on $N$. Hence we get that:
\begin{proposition}\label{Proposition3} $(N,\nabla,h^{N})$ is a K-V manifold.
\end{proposition}

To show this proposition we need the following lemma:
\begin{lemma}\label{Lemma8}
	Let $N$ be a K-V tranversal of $(M,\nabla,h)$. Then $TN$ is a Poisson transversal submanifold of $(TM,\Pi)$. 
\end{lemma}
\begin{proof}
	Let $x\in M$, $\alpha,\beta\in\Gamma(T^{*}M)$ and $X\in\Gamma(TM)$ such that $\alpha_{x}\in T_{x}N^{\circ}$. Let $u\in T_{x}N$. For all $Z\in\Gamma(TN)$ 
	\begin{equation*}
	\prec (\alpha^{h})_{u},(Z^{h})_{u}\succ=\prec \alpha_{x},Z_{x}\succ=0,
	\end{equation*}
	and 
	\begin{equation*}
	\prec (\alpha^{v})_{u},(Z^{v})_{u}\succ=\prec \alpha_{x},Z_{x}\succ=0
	,\end{equation*}
	hence $(\alpha^{h})_{u},(\alpha^{v})_{u}\in T_{u}(TN)^{\circ}$. On the other hand there exist $Y\in\Gamma(TN)$ and $\gamma\in\Gamma(T^{*}M)$ such that $\gamma_{x}\in T_{x}N^{\circ}$ and $X_{x}=Y_{x}+\gamma_{x}^{\#}$. Hence we have 
	\begin{eqnarray*}
		\prec (\beta^{v})_{u},(X^{v})_{u}\succ &=& \prec \beta_{x},X_{x}\succ\\
		&=& \prec \beta_{x},Y_{x}+\gamma_{x}^{\#}\succ\\
		&=& \prec (\beta)^{v}_{u},(Y^{v})_{u}+(\gamma_{x}^{\#})_{u}^{v}\succ\\
		&=& \prec (\beta)^{v}_{u},(Y^{v})_{u}+\Pi_{\#}(\gamma^{h}_{u})\succ.
	\end{eqnarray*}
	This implies that $(X^{v})_{u}=(Y^{v})_{u}+\Pi_{\#}(\gamma^{h}_{u})$. Similarly we get that $(X^{h})_{u}=(Y^{h})_{u}-\Pi_{\#}(\gamma^{v}_{u})$. Hence the equality $T_{u}(TM)=T_{u}(TN)+\Pi_{\#}(T_{u}(TN)^{\circ})$ follows from the fact that $T_{u}(TM)=\mathrm{span}\{(X^{v})_{u},(X^{h})_{u}| X\in\Gamma(TM)\}$. 
\end{proof}
\begin{proof}[Proof of Proposition \ref{Proposition3}]
	Let $\Pi^{N}$ be the skew-symmetric bivector field associated to the pair $(\nabla,h^{N})$, one can see also that $\Pi^{N}$ coincide with the bundle map given by the following composition maps
	\begin{equation*}
	T^{*}(TN)\stackrel{Tp^{*}}{\longrightarrow} T^{*}_{TN}(TM)\stackrel{\Pi_{\#}}{\longrightarrow}T_{TN}(TM)\stackrel{Tp}{\longrightarrow}T(TN).
	\end{equation*}
	According to the Lemma \ref{Lemma8} and Proposition $2.13$ in \cite{Fernandes}, $\Pi^{N}$ is a Poisson teansor field on $TN$, hence we get that $(\nabla,h^{N})$ is a K-V structure on $N$.  .
\end{proof}

It is important to note that for a K-V transversal $N$ in $(M,\nabla,h)$, with induced K-V structure $h^{N}$, the inclusion map $\iota:(N,\nabla,h^{N})\hookrightarrow(M,\nabla,h)$ is not a K-V map (Unless is an open set in $M$). This will be clear in the next example.
\begin{example}\label{Example5}
	Take $M=(\mathbb{R}^{3},\nabla,h)$ endowed with its canonical affine structure and $h=x\partial_{x}\otimes\partial_{x}+y\partial_{y}\otimes\partial_{y}$.
	Clearly $(M,\nabla,h)$ is a K-V structure and $N:=\{(0,0,z)|\text{ }z\in\mathbb{R}\}\subset (M,\nabla,h)$ is a K-V transversal, however the induced K-V bivector field on $N$ vanish identically.
\end{example}

K-V transversals behave functorially under pullbacks by K-V maps. This turns out to be a very useful property:
\begin{proposition}\label{Proposition4}
	Let $F:(M^{1},\nabla^{1},h^{1})\rightarrow (M^{2},\nabla^{2},h^{2})$ be a K-V map and let $N^{2}\subset M^{2}$ be a K-V transversal. Then $F$ is transverse to $N^{2}$ and $N^{1}:=F^{-1}(N^{2})$ is a K-V transversal in $M^{1}$. Moreover, $F$ restricts to a K-V map between the induced K-V structure on $N^{1}$ and $N^{2}$.
\end{proposition}
\begin{proof} Consider $x\in N^{1}$ and let $y=F(x)\in N^{2}$. Since $F$ is a K-V map we have for all $\alpha\in T^{*}_{y}M^{2}$
	\begin{equation}\label{eq trans}
	\alpha^{\#_{2}}=T_{x}F((T_{x}F^{*}\alpha)^{\#_{1}}).
	\end{equation}
	Therefore $h^{2}_{\#}(T^{*}_{y}M^{2})\subset T_{x}F(T_{x}M^{1})$. Since $N^{2}$ is a K-V transversal this implies that 
	\begin{equation*}
	T_{y}M^{2}=T_{y}N^{2}+h_{\#}^{2}(T_{y}^{*}M^{2})=T_{y}N^{2}+T_{x}F(T_{x}M^{1}).
	\end{equation*}
	This show that $F$ is transverse to $N^{2}$. In particular, $N^{1}$ is a submanifold of $M^{1}$.\\
	The affinity of $N^{1}$ follows from Theorem 2 in (\cite{Linden}). Let $v\in T_{x}M^{1}$, and decompose $T_{x}F(v)=u+\alpha^{\#_{2}}$, with $u\in T_{y}N^{2}$ and $\alpha\in (T_{y}N^{2})^{\circ}$. Then $F^{*}\alpha\in (T_{x}N^{1})^{\circ}$, and by (\ref{eq trans}), the vector $w:=v-(T_{x}F^{*}\alpha)^{\#_{1}}$ is mapped by $T_{x}F$ to $u$. Hence $w\in T_{x}N^{1}$. This shows that 
	\begin{equation*}
	v=w+(T_{x}F^{*}\alpha)^{\#_{1}}\in T_{x}N^{1}+h_{\sharp}^{1}((T_{x}N^{1})^{\circ}).
	\end{equation*}
	Therefore $N^{1}$ is a K-V transversal.
	
	According to the Lemma \ref{Lemma8} we get that  $TN^{2}\subset(TM^{2},\Pi^{2})$ is a Poisson transversal. From Proposition $2.20$ in \cite{Fernandes} it follow that $TF:(TN^{1},\Pi^{N^{1}})\rightarrow(TN^{2},\Pi^{N^{2}})$ is a Poison map. Therefore  $F:(N^{1},\nabla^{1},h^{N^{1}})\rightarrow(N^{2},\nabla^{2},h^{N^{2}}) $ is a K-V map. 
\end{proof}
\begin{corollary}\label{Corollary4}
	Let $(M,\nabla,h)$ be a K-V manifold, and $(N^{1},\na^{1},h^{1})\subset M$ be a K-V transversal and let $(N^{2},\na^{2},h^{2})\subset M$ be a K-V submanifold. Then $N^{1}$ and $N^{2}$ intersect transversally, $N^{1}\cap N^{2}$ is a K-V transversal in $N^{2}$ and a K-V submanifold of $N^{1}$, and the two induced K-V structures on $N^{1}\cap N^{2}$ coincide.
\end{corollary}

Now we give the relation between K-V transversals and the affine foliation.
\begin{proposition}
	Let $(M,\na,h)$ be a K-V manifold with affine foliation $\mathcal{F}$. An affine submanifold $N\subset M$ is a K-V transversal if and only if for all $L\in\mathcal{F}$, the intersection $L\cap N$ is a K-V submanifold of $L$. Hence, the affine foliation of $\left(N,\na,h^{N}\right)$ consists of the connected compenents of the intersection $L\cap N$. 
\end{proposition}
\begin{proof}
	The condition for affine submanifold $N\subset M$ to be a K-V transversal is 
	\begin{equation*}
	T_{N}M=TN\oplus h_{\#}\left(TN^{\circ}\right).
	\end{equation*}
	This condition is equivalent to have both the following conditions satisfied:
	\begin{itemize}
		\item[$(i)$] $T_{N}M=TN+h_{\#}\left(TN^{\circ}\right)$.
		\item[$(ii)$] $TN\cap h_{\#}\left(TN^{\circ}\right)=0$.
	\end{itemize}
	Condition $(i)$ says that $N$ is transverse to the affine leaves and condition $(ii)$ (provided $(i)$ is satisfied) says that the kernel of the pullback of $g_{L}$ to $L\cap N$ is trivial, so $L\cap N$ is a K-V submanifold of $L$.
\end{proof}

We finish this section by the following example:
\begin{example} Consider $\mathbb{R}^{n}$ endowed with its canonical affine structure $\na$ and the standard K-V tensor:
	\begin{equation*}
	h=\displaystyle\sum_{i=}^{n}\frac{\partial}{\partial x_{i}}\otimes\frac{\partial}{\partial x_{i}}.
	\end{equation*}
	Let $$F:(\mathbb{R}^{n},\na,h)\rightarrow(\mathbb{R}^{m},\na,h)$$
	be an affine map.Clearly $F$ is a K-V map and for every regular value $x\in\mathbb{R}^{m}$ of $F$ the affine submanifold $F^{-1}({x})\subset\mathbb{R}^{n}$ is a K-V transversal. 
\end{example}
\section{Coisotropic Koszul-Vinberg submanifolds}\label{sec4}

Let $(P,\pi)$ be a Poisson manifold. A submanifold $N\subset P$ is called coisotropic submanifold if and only if $\pi_{\#}(TN^{\circ})\subset TN$. Now we have a similar statement for K-V manifolds, we consider a K-V manifold $(M,\nabla,h)$ and $N$ be an affine submanifold of $M$.
\begin{proposition}\label{Proposition5} 
	The following assertions are equivalent
	\begin{enumerate}
		\item $h_{\#}(TN^{\circ})\subset TN$.
		\item $TN$ is a coisotropic submanifold of $(TM,\Pi)$.
	\end{enumerate}
\end{proposition}
\begin{proof}
	Let $\alpha\in\Omega^{1}(M)$ then $\alpha_{|_{TN}}=0$ if and only if $\alpha^{h}_{|_{T(TN)}}=\alpha^{v}_{|_{T(TN)}}=0$. Hence the desired equivalence follows from equation $(1.4)$.
\end{proof}
\begin{definition}\label{Definition3}
	A coisotropic K-V submanifold $N\subset M$ is an affine submanifold such that $h_{\#}(TN^{\circ})\subset TN$.
\end{definition}

Clearly K-V submanifolds is a subclass of coisotropic K-V submanifolds.
\begin{proposition}\label{Proposition6}
	Let $(M,\nabla,h)$ be a K-V manifold. For an affine closed submanifold $N\subset M$ the following conditions are equivalent:
	\begin{enumerate}
		\item[$(i)$] $N$ is a coisotropic K-V submanifold;
		\item[$(ii)$] For every $f,g\in\mathcal{J}(N)$ where $\mathcal{J}(N)$ is the vanishing ideal we have $X_{f}(g)\in\mathcal{J}(N)$;
		\item[$(iii)$] For every $f\in\mathcal{J}(N)$ the Hamiltonian vector field $X_{f}$ is tangent to $N$.
	\end{enumerate}
\end{proposition}
\begin{proof} 
	$(i)\Longrightarrow(ii)$ Let $f,g\in\mathcal{J}(N)$ clearly for any $x\in N$ we have $d_{x}f,d_{x}g\in T_{x}N^{\circ}$. Now let $N\subset M$ be a coisotropic K-V submanifold then $h(x)(d_{x}f,d_{x}g)=0$, so $h(df,dg)\in\mathcal{J}(N)$.
	
	$(ii)\Longrightarrow(iii)$ Assume that $\mathcal{J}(N)$ is stable under $h$. Let $f,g\in\mathcal{J}(N)$, for all  $x\in N$
	\begin{equation*}
	X_{f}(g)(x)=h(x)(d_{x}f,d_{x}g)=0.
	\end{equation*}
	The closedness of the submanifold $N$ implies that $X_{f}$ is tangent to $N$.
	
	$(iii)\Longrightarrow(i)$  Again the closedness of the submanifold $N$ implies that $T_{x}N^{\circ}$ is generated  by elements $d_{x}f$ where $f\in\mathcal{J}(N)$. so we conclude that for any $\alpha,\beta\in TN^{\circ}$, $h(\alpha,\beta)=0$. Therefore, $N$ is coisotropic.
\end{proof}
\begin{proposition}\label{Proposition7}
	Let $F:(M^{1},\nabla^{1},h^{1})\rightarrow(M^{2},\nabla^{2},h^{2})$ be a K-V map and assume that $F$ is transverse to a coisotropic K-V submanifold $N^{2}\subset M^{2}$. Then $F^{-1}(N^{2})\subset M^{1}$ is a coisotropic K-V submanifold.
\end{proposition}
\begin{proof} The result follows by applying Propostion $2.34$ in (\cite{Fernandes}) to the Poisson map $TF:(TM^{1},\Pi^{1})\rightarrow(TM^{2},\Pi^{2})$, with the affinity of $F^{-1}(N^{2})$ which is garanted by Theorem 2 in (\cite{Linden}). 	
\end{proof}

There is one more important property of coisotropic objects and which shows their relevance in K-V geometry. In order to express it, we introduce the following notation: If $(M^{1},\nabla^{1},h^{1})$ and $(M^{2},\nabla^{2},h^{2})$ are K-V manifolds we denote by $(M^{1}\times\overline{M}^{2},\nabla,h)$ the K-V manifold such that the canonicals projection $p_{1}:(M^{1}\times\overline{M}^{2},\nabla,h)\rightarrow(M^{1},\nabla^{1},h^{1})$ and  $p_{2}:(M^{1}\times\overline{M}^{2},\nabla,h)\rightarrow(M^{2},\nabla^{2},-h^{2})$ are K-V maps.
\begin{proposition}\label{Proposition8} 
	For a smooth map $F:(M^{1},\nabla^{1},h^{1})\rightarrow(M^{2},\nabla^{2},h^{2})$ the following conditions are equivalent:
	\begin{itemize}
		\item[$(i)$] $F$ is a K-V map.
		\item[$(ii)$] $\mathrm{Graph}(F)\subset M^{1}\times\overline{M}^{2}$ is a coisotropic K-V submanifold. 
	\end{itemize}
\end{proposition}

To prove this Proposition we need the following lemma which is a generalization of the following fact: A map $F:\mathbb{R}^{n}\rightarrow\mathbb{R}^{m}$ is affine if and only if its graph is an affine subspace of $\mathbb{R}^{n+m}$. Now let $(M^{1},\nabla^{1})$ and $(M^{2},\nabla^{2})$ be two affine manifolds and $F:(M^{1},\nabla^{1})\rightarrow(M^{2},\nabla^{2})$ is a smooth map.
\begin{lemma}\label{Lemma6}
	$F$ is an affine map if and only if its graph is an affine submanifold of $(M^{1}\times M^{2},\nabla^{1}\oplus\nabla^{2})$.
\end{lemma}
\begin{proof}
	Let $\tilde{\gamma}:[0,1]\rightarrow \mathrm{Graph}(F)$, $\tilde{\gamma}=(\gamma,F(\gamma))$ where $\gamma$ is a curve on $M^{1}$ such that $\gamma(0)=p$ and $\gamma(1)=q$ hence for any $u\in T_{p}M^{1}$ we have 
	\begin{equation*}
	\tau^{\tilde{\gamma}}(u,T_{p}F(u))=(\tau^{\gamma}(u),\tau^{F(\gamma)}(T_{p}F(u)))
	\end{equation*} 
	where $\tau^{\tilde{\gamma}}$ is the parallel transport along the curve $\tilde{\gamma}$ seen as a curve on $M^{1}\times M^{2}$. Hence $T_{q}F(\tau^{\gamma}(u))=\tau^{F(\gamma)}(T_{p}F(u))$ if and only if $\tau^{\tilde{\gamma}}(u,T_{p}F(u))\in T_{(q,F(q))}\mathrm{Graph}(F)$. This show that $F$ is an affine map if and only if commute with parallel transport if and only if $\mathrm{Graph}(F)$ is an affine submanifold of $(M^{1}\times M^{2},\nabla^{1}\oplus\nabla^{2})$. 
\end{proof}
\begin{proof}[Proof of Proposition \ref{Proposition8}]
	Follow directly from Lemma \ref{Lemma6}.
\end{proof}

Now let's explore more proprieties of coisotropic K-V subamanifold. For that we endowed the space of section of the vector bundle $TN^{\circ}$ with the product $\bullet$ given by:
\begin{equation*}
\alpha\bullet\be=\displaystyle\mathcal{D}_{\al}\be.
\end{equation*}
\begin{proposition}
	Let $N\subset (M,\na,h)$ be a coisotropic submanifold. Then $(TN^{\circ},N,\bullet,\rho)$ is a left symmetric algebroid, where $\rho$ is the restriction of $h_{\#}$ to the subbundle $TN^{\circ}\subset T_{N}^{*}M$. Moreover fro any $x\in N$ the vector space $\mathfrak{g}_{x}=\ker \rho_{x}$ is a commutative associative algebra.
\end{proposition}
\begin{proof}
	What we need to show is that the product $\bullet$ is well defined the other assertions is a direct consequence of this fact. Let $X$ be a vector field on $M$ which is tangent to $N$ and $\al,\be\in\Gamma(TN^{\circ})$ then we have
	\begin{eqnarray*}
		\prec\mathcal{D}_{\al}\be,X\succ&=&-X.(\prec\al,\be^{\#}\succ)+\prec\al,\na_{X}\be^{\#}\succ+\prec\be,\na_{X}\al^{\#}\succ \\
		& &+h_{\#}(\al).(\prec\be,X\succ)-\prec\be,\na_{\al^{\#}}X\succ
	\end{eqnarray*}
	using the affinity of $N$ together with the condition $h_{\#}(TN^{\circ})\subset TN$ we get that $\prec\mathcal{D}_{\al}\be,X\succ=0$ hence $\mathcal{D}_{\al}\be\in\Gamma(TN^{\circ})$.
\end{proof}
\begin{example}$\ $
	\begin{enumerate}
		\item $N$ is a K-V submanifold if and only if $\rho=0$.
		\item If $x \in M$ is a point at which the K-V structure is zero, then $\{x\}$ is coisotropic.
		\item Let $\iota:\mathcal{H}\hookrightarrow\mathcal{A}$ be a subalgebra of the associative commutative algebra $\mathcal{A}$. Then $(\iota^{*})^{-1}(\{0\})\subset\mathcal{A}^{*}$ is a cositropic K-V submanifold where $\mathcal{A}^{*}$ is endowed with its canonical linear K-V structure.
	\end{enumerate}
\end{example}

Recall that two submanifolds $N^{1},N^{2}\subset M$ are said to have a clean intersection if $N^{1}\cap N^{2}$ is a submanifold of $M$ and $T(N^{1}\cap N^{2})=TN^{1}\cap TN^{2}$. A submanifold has clean intersection with a foliation if it intersects cleanly every leaf of the foliation.
\begin{proposition}
	Let $N$ be an affine submanifold of a K-V manifold $(M,\na,h)$ which has
	clean intersection with its affine foliation $\mathcal{F}$. Then $N$ is a coisotropic submanifold of $(M,\na,h)$ if and only if for each affine leaf $L\in\mathcal{F}$ the intersection $L\cap N$ is
	a coisotropic submanifold of $N$.
\end{proposition}
\begin{proof}
	Assume that $N\subset M$ is an affine submanifold which is transverse to the affine
	foliation $\mathcal{F}$. This means that for each $L\in\mathcal{F}$ the inclusion $\iota:L\hookrightarrow M$ is transverse
	$N$. Now:
	\begin{itemize}
		\item[$(a)$] If $N$ is coisotropic in $M$, it follows that $\iota^{-1}(N)=L\cap N$ is coisotropic in $N$, since
		the inclusion is a K-V map.
		\item[$(b)$] If $L\subset N$ is coisotropic in $L$, then we have:
		$$h_{\#}^{L}(T(L\cap N)^{\circ})\subset T(L\cap N)=TL\cap TN ,$$
	\end{itemize}
	where the annihilator is in $T^{*}L$. It follows that for for any a $\al\in\Omega^{1}(M)$ such that $\alpha_{|_{T(L\cap N)}}=0$, we have:
	$\al^{\#}\in\Gamma(TN)$.
	But $(TL\cap TN)^{\circ}=TL^{\circ}+TN^{\circ}$, so we conclude that $h_{\#}(TN^{\circ})\subset TN$, which means that $N$ is coisotropic.
\end{proof}

\bibliographystyle{amsplain}

\end{document}